\documentclass[12pt]{amsart}
\usepackage[utf8]{inputenc}
\usepackage{amsxtra,amssymb,amsthm,amsmath,amscd,mathrsfs, epsfig, eufrak}
\usepackage{amsfonts, amscd, amsmath, mathrsfs, amssymb, amsthm, amsxtra, bbding, epsfig, eucal, graphicx, latexsym, url, mathbbol, bbold}
\usepackage[all]{xy}
\usepackage{tikz-cd} 
\usepackage{tikz}
\usepackage[normalem]{ulem}
\usepackage{soul}
\usepackage{color}

\setstcolor{red}
\makeatletter
\@namedef{subjclassname@2010}{%
 \textup{2010} Mathematics Subject Classification}
\makeatother

\def \P {{\mathbb P}}
\def \Q {{\mathbb Q}}

\def \Z {{\mathbb Z}}

\def \d {\,{\rm d}}
\def\re{{\Re e\,}}

\def \sset {{\smallsetminus }}

\def\leq{\leqslant}
\def\geq{\geqslant}
\def\le{\leqslant}
\def\ge{\geqslant}

\theoremstyle{plain}
\newtheorem{theorem}{Theorem}[section]
\newtheorem{proposition}{Proposition}[section]
\newtheorem{lemma}[proposition]{Lemma}

\theoremstyle{remark}

\numberwithin{equation}{section}
\addtocounter{footnote}{1}

\begin{document}

\title[On a senary quartic form]
{On a senary quartic form}
\author{Jianya Liu, Jie Wu \& Yongqiang Zhao}

\address{%
Jianya Liu
\\
School of Mathematics
\\
Shandong University
\\
Jinan
\\
Shandong 250100
\\
China} \email{jyliu@sdu.edu.cn}

\address{%
Jie Wu
\\
CNRS UMR 8050\\
Laboratoire d'analyse et de math\'ematiques appliqu\'ees\\
Universit\'e Paris-Est Cr\'eteil\\
94010 Cr\'eteil cedex\\
France
}
\email{jie.wu@math.cnrs.fr}

\address{%
Yongqiang Zhao
\\
Westlake University
\\
School of Science
\\
Shilongshan Road
\\
Cloud Town
\\
Xihu District
\\
Hangzhou
\\
Zhejiang 310024
\\ 
China} 
\email{yzhao@wias.org.cn}

\date{\today}

\begin{abstract} 
We count rational points of bounded height on  the 
non-normal senary 
quartic hypersurface 
$x^4=(y_1^2 + \cdots + y_4^2)z^2 $ in the spirit of Manin's conjecture. 
\end{abstract}


\maketitle   	
   	
\section{Introduction}
Recently, we~\cite{LWZ2017} proved Manin's conjecture for singular cubic hypersurfaces 
\begin{equation}\label{def:Sn}
x^3=(y_1^2+\cdots+y_{n}^2)z, 
\end{equation}
where $n$ is a positive multiple of 4.  
In this short note, we show that our method used in ~\cite{LWZ2017} also works for  higher degree forms like 
\begin{equation}
x^m=(y_1^2+\cdots+y_{n}^2)z^{m-2}, 
\end{equation}
where $n \geq 4 $ and $m \geq 4$. 
To illustrate, we  establish an asymptotic formula for the number of rational points 
of bounded height on the quartic hypersurface 
\begin{equation}\label{def:Smn}
Q: \,\, x^4=(y_1^2+ y_2^2 + y_3^2 + y_4^2)z^{2},
\end{equation}
in the spirit of Manin's conjecture. 

It is easy to see that the subvariety  $x=z=0$ of $Q$ 
already contains $\gg B^4$ rational points 
with $ |x|\leq B,  |z|\leq B,$ and  $|y_j|\leq B$ with   
$1\leq j \leq 4$,  which is predominant and  is much larger 
than the heuristic prediction that is of order $B^2$.   
One therefore counts rational points on the complement subset  $U=Q\sset \{x=z=0 \}$.
Let $H$ be the height function  
$$
H(x: y_1: \dots : y_4 :z)= \max\big\{|x|, \sqrt{y_1^2 + \cdots + y_4 ^2}, |z|\big\} 
$$
for $(x, \, y_1, \dots, y_4, \, z)=1$.  
Let $B$ be a large integer, and define 
$$ 
N_U(B) := \big|\big\{ (x: y_1: \dots : y_4 :z) \in U  :    
H(x: y_1: \dots : y_4 :z)\leq B\big\}\big|. 
$$
This counts rational points in $U$ whose height is bounded by $B$, and the 
aim of this note is obtain an asymptotic formula for it. To this end, we need to 
understand in advance a similar quantity 
\begin{align*}
N_U^*(B) 
& := \sum_{\substack{1\le |x|\le B, \, 
1\le y_1^2+\cdots+y_4^2\le B^2,  \, |z|\le B \\ x^4=(y_1^2+\cdots+y_4^2)z^2}} 1.
\end{align*}
One sees, in $N^*_U(B)$, that the co-prime condition $(x, \, y_1, \dots, y_4, \, z)=1$ in $N_U(B)$ 
is relaxed. Our main result is as follows.

\begin{theorem}\label{thm1} 
As $B\to\infty$, we have 
\begin{align}
N_U(B)
& = \mathcal{C}_4 B^3\log B \, \bigg\{1+O\bigg(\frac{1}{\sqrt[4]{\log B}}\bigg)\bigg\},
\label{eq:N4B}
\\
N_U^*(B)
& = \mathcal{C}_4^* B^3\log B \, \bigg\{1+O\bigg(\frac{1}{\sqrt[4]{\log B}}\bigg)\bigg\}
\label{eq:N4*B}
\end{align}
with $\mathcal{C}_4 := \frac{192}{5\zeta(3)} \mathscr{C}_4$ and
$\mathcal{C}_4^* := \frac{192}{5} \mathscr{C}_4$,
where $\mathscr{C}_4$ is defined as in \eqref{def:LeadingCoefficient} below, 
and $\zeta$ is the Riemann zeta-function. 
\end{theorem}

We note that the exponent of $B$ in the main terms of the above theorem is $3$ 
instead of  $2$ as predicted by the usual heuristic.  
This phenomenon may be explained by the fact that the hypersurface $Q$ is not normal. 

It is easy to check that $Q$ has an obvious quadric bundle structure given by  
\begin{equation}\label{para}
Q_{[a: b]} :  \ \ 
\left\{
\begin{array}{ll}
b^2x^2=a^2(y_1^2+ y_2^2 + y_3^2 + y_4^2),  \\ 
ax- by=0, 
\end{array}
\right.
\end{equation}
and $\{Q_{[a: b]}\}$ covers $Q$ as long as $[a: b]$ goes thorough ${\P}^1(\Q)$. 
From this, it is possible to interpret Theorem~\ref{thm1}   
in the framework of the generalized Manin's conjecture 
by Batyrev and Tschinkel~\cite{BT}, as was done in the work of  
de la Bret\`eche, Browning, and Salberger~\cite{BBS}. However, we will not pursue such an explanation here. 
The only sole purpose of this short note is to show that our method used in ~\cite{LWZ2017} also works for  
higher degree forms $Q$.

Finally, we remark that using the method in our joint paper ~\cite{BLWZ} 
with de la Bret\`eche,  one can get power-saving error terms in 
Theorem~\ref{thm1}, which we will not pursue here.

\section{Outline of the proof of Theorem~\ref{thm1}} 

Denote by $r_4(d)$ the number representations of a positive integer $d$ as the sum of four 
squares : $d=y_1^2+\cdots+y_4^2$ with $(y_1, \dots, y_4)\in \Z^4$.
It is well-known (cf. \cite[(3.9)]{Grosswald1985}) that 
\begin{equation}\label{def:r4*}
r_4(d) = 8r_4^*(d)
\quad\text{with}\quad
r_4^*(d) := \sum_{\ell\mid d, \, \ell\not\equiv 0 ({\rm mod}\,4)} \ell.  
\end{equation}
Let $\mathbb{1}_{\square}(n)$ be the characteristic function of squares.
In view of the above, we can write
\begin{equation}\label{decomposition:N4B}
N_U^*(B)
= 32 \bigg\{\sum_{n\le B} \sum_{\substack{d\mid n^4\\ d\le B^2}} r_4^*(d) 
\mathbb{1}_{\square}\bigg(\frac{n^4}{d}\bigg)
- \sum_{n\le B} \sum_{\substack{d\mid n^4\\ d<n^4/B^2}} r_4^*(d) 
\mathbb{1}_{\square}\bigg(\frac{n^4}{d}\bigg)\bigg\}.
\end{equation}
Hence to prove \eqref{eq:N4*B} in Theorem~\ref{thm1}, it is sufficient to establish 
asymptotic formulae for the following two quantities
\begin{equation}\label{def:Sxy}
S(x, y) := \sum_{n\le x} \sum_{\substack{d\mid n^4\\ d\le y}} r_4^*(d) 
\mathbb{1}_{\square}\bigg(\frac{n^4}{d}\bigg),
\qquad
T(B) := \sum_{n\le B} \sum_{\substack{d\mid n^4\\ d<n^4/B^2}} r_4^*(d) 
\mathbb{1}_{\square}\bigg(\frac{n^4}{d}\bigg).
\end{equation}

For $S(x, y)$, our result is as follows.

\begin{theorem}\label{thm2}
Let $\varepsilon>0$ be arbitrary. We have 
\begin{equation}\label{Evaluation:Sxy}
S(x, y)  
= xy \big(4P(\psi) + \tfrac{3}{2}P'(\psi)\big) 
+ O_{\varepsilon}\big(x^{\frac{5}{4}} y^{\frac{7}{8}} 
+ x^{\frac{1}{2}+\varepsilon} y^{\frac{9}{8}}\big)
\end{equation}
uniformly for $x^3\ge y\ge x\ge 10$,
where $\psi:=\log x - \tfrac{1}{4}\log y$ and $P(t)$ is a quadratic 
polynomial, defined as in \eqref{def:Pt} below.
In particular, for any fixed $\eta\in (0, 1]$ we have
\begin{equation}\label{Cor:Sxy}
S(x, y)  
= 4 \mathscr{C}_4 xy 
\bigg(\log x-\frac{1}{4}\log y\bigg) \bigg\{1 + O\bigg(\frac{1}{(\log x)^{\eta}}\bigg)\bigg\}
\end{equation}
uniformly for $x\ge 10$ and $x^2(\log x)^{-8(1-\eta)}\le y\le x^3$,
where 
\begin{equation}\label{def:LeadingCoefficient}
\mathscr{C}_4
:= \frac{23}{150} \zeta(5) \prod_{p} \bigg(1
+\frac{1}{p}
+\frac{2}{p^2}
+\frac{2}{p^3}
+\frac{1}{p^4}
+\frac{1}{p^5}\bigg)\bigg(1-\frac{1}{p}\bigg) 
\end{equation}
is the leading coefficient of $P(t)$.
\end{theorem}

Now we turn to analyze $T(B)$ which is more difficult, since the range of 
its second summation depends on the variable $n$ of the first summation. 
Thus Theorem~\ref{thm2} does not apply to $T(B)$ directly.  
In \S\ref{PfThm3} we show that Theorem~\ref{thm2} together with 
delicate analysis is sufficient to establish the following result. 

\begin{theorem}\label{thm3}
As $B\to\infty$, we have
\begin{equation}\label{Evaluation:TB}
T(B)
= \frac{2}{5} \mathscr{C}_4 B^3 \log B
\bigg\{1+O\bigg(\frac{1}{\sqrt[4]{\log B}}\bigg)\bigg\}, 
\end{equation}
where 
$\mathscr{C}_4$ is as in \eqref{def:LeadingCoefficient} above. 
\end{theorem}

As in \cite{Breteche1998, LWZ2017}, we shall firstly establish an asymptotic formula for the quantity
\begin{equation}\label{def:MXY}
M(X, Y) := \int_1^Y \int_1^X S(x, y) \d x \d y. 
\end{equation}
by applying the method of complex integration. 
Then we derive the asymptotic formula \eqref{Evaluation:Sxy} for $S(x, y)$ in Theorem~\ref{thm2} 
by the operator $\mathscr{D}$ defined below. 
Let $\mathscr{E}_k$ be the set of all functions of $k$ variables.
Define the operator $\mathscr{D}: \mathscr{E}_2\to \mathscr{E}_4$ 
by
\begin{equation}\label{def:Tf}
(\mathscr{D}f)(X, H; Y, J) := f(H, J) - f(H, Y) - f(X, J) + f(X, Y).
\end{equation} 
The next lemma summarises all properties of $\mathscr{D}$ needed later.

\begin{lemma}\label{Operator/2}
\par
{\rm (i)}
Let $f\in \mathscr{E}_2$ be a function of class $C^3$. Then we have
$$
(\mathscr{D}f)(X, H; Y, J)
= (J-Y)(H-X) \bigg\{\frac{\partial^2f}{\partial x\partial y}(X, Y) + O\big(R(X, H; Y, J)\big)\bigg\}
$$
for $X\le H$ and $Y\le J$, where
$$
R(X, H; Y, J) 
:= (H-X) 
\max_{\substack{X\le x\le H\\ Y\le y\le J}} \bigg|\frac{\partial^3f}{\partial x^2\partial y}(x, y)\bigg|
+ (J-Y)\max_{\substack{X\le x\le H\\ Y\le y\le J}} \bigg|\frac{\partial^3f}{\partial x\partial y^2}(x, y)\bigg|.
$$
\par
{\rm (ii)}
Let $S(x, y)$ and $M(X, Y)$ be defined as in \eqref{def:Sxy} and \eqref{def:MXY}.
Then 
$$
(\mathscr{D}M)(X-H, X; Y-J, Y)\le HJS(X, Y)\le (\mathscr{D}M)(X, X+H; Y, Y+J)
$$
for $H\le X$ and $J\le Y$.  
\end{lemma}

The next elementary estimate (\cite[Lemma 6(i)]{Breteche1998} or \cite[Lemma 4.3]{LWZ2017}) 
will also be used several times in the paper. 

\begin{lemma}\label{Lem3.2}
Let $1\le H\le X$ and $|\sigma|\le 10$.
Then for any $\beta\in [0, 1]$, we have
\begin{equation}\label{Lem3.2_Eq_A}
\big|(X+H)^{s} - X^{s}\big|\ll X^{\sigma} ((|\tau|+1)H/X)^{\beta},
\end{equation}
where the implied constant is absolute.
\end{lemma}

\section{Dirichlet series associated with $S(x, y)$}\label{Dirichlet}

In view of the definition of $S(x, y)$ in \eqref{def:Sxy}, we define 
the double Dirichlet series 
\begin{equation}\label{def:Fsw}
\mathcal{F}(s, w) := \sum_{n\ge 1} n^{-s} \sum_{d\mid n^4} 
d^{-w} r_4^*(d) \mathbb{1}_{\square}\bigg(\frac{n^4}{d}\bigg)
\end{equation}
for $\re s>5$ and $\re w>0$. 
The next lemma states that the function 
$\mathcal{F}(s, w) $ enjoys a nice factorization formula. 

\begin{lemma}\label{Lem:Fsw}
For $\min_{0\le j\le 2} \re (s+2jw-2j)>1$, 
we have
\begin{equation}\label{Expression:Fsw}
\mathcal{F}(s, w) = \prod_{0\le j\le 2} \zeta(s+2jw-2j) \mathcal{G}(s, w),
\end{equation}
where $\mathcal{G}(s, w)$ is an Euler product, given by \eqref{def:Gpsw}, \eqref{def:G2sw} and \eqref{def:Gsw} below.
Further, for any $\varepsilon>0$ and for $\min_{0\le j\le 2} \re (s+2jw-2j)\ge \tfrac{1}{2}+\varepsilon$, 
$\mathcal{G}(s, w)$ converges absolutely and
\begin{equation}\label{UB:Gsw}
\mathcal{G}(s, w)\ll_{\varepsilon} 1.
\end{equation}
\end{lemma}

\begin{proof}
Since the functions $r_4^*(d)$ and 
$n^{-s} \sum_{d\mid n^4} d^{-w} r_4^*(d)\mathbb{1}_{\square}(n^4/d)$ are multiplicative,
for $\re s>5$ and $\re w>0$ we can write the Euler product
\begin{align*}
\mathcal{F}(s, w)
& = \prod_p \sum_{\nu\ge 0} p^{-\nu s} \sum_{\substack{0\le \mu\le 2\nu}} 
p^{-2\mu w} r_4^*(p^{2\mu}) 
= \prod_p \mathcal{F}_p(s, w).
\end{align*}
In the above computations, speacial attention should be paid to the effect of 
the function $\mathbb{1}_{\square}$.  
The next is to simplify each $\mathcal{F}_p(s, w)$. To this end, we recall 
\eqref{def:r4*} so that 
\begin{equation}\label{def:r4*pmu}
r_4^*(p^{\mu}) = \dfrac{1-p^{\mu+1}}{1-p}  
\quad
(p>2),
\qquad
r_4^*(2^{\mu}) = 3 
\end{equation}
for all integers $\mu\ge 1$. On the other hand, a simple formal calculation shows
\begin{equation}\label{Formal_Calcul_1}
\begin{aligned}
\sum_{\nu\ge 0} x^{\nu} \sum_{0\le \mu\le 2\nu} y^{2\mu} \frac{1-z^{2\mu+1}}{1-z} 
& = \frac{1}{1-z} \sum_{\nu\ge 0} x^{\nu} 
\bigg(\frac{1-y^{4\nu+2}}{1-y^2} - z\frac{1-(yz)^{4\nu+2}}{1-y^2z^2}\bigg)
\\\noalign{\vskip 1mm}
& = \frac{1+xy^2(1+z+z^2)+xy^4(z+z^2+z^3)+x^2y^6z^3}{(1-x)(1-xy^4)(1-xy^4z^4)}
\end{aligned}
\end{equation}
and
\begin{equation}\label{Formal_Calcul_2}
\begin{aligned}
1 + \sum_{\nu\ge 1} x^{\nu} 
\bigg(1 + a \sum_{1\le \mu\le 2\nu} y^{2\mu}\bigg)
& = 1 + \sum_{\nu\ge 1} x^{\nu} 
\bigg(1 + a \frac{y^2-y^{4\nu+2}}{1-y^2}\bigg)
\\\noalign{\vskip 0,5mm}
& = \frac{1+axy^2+(a-1)xy^4}{(1-x)(1-xy^4)}\cdot
\end{aligned}
\end{equation}
When $p>2$, in view of \eqref{def:r4*pmu},
we can apply \eqref{Formal_Calcul_1} with $(x, y, z) = (p^{-s}, p^{-w}, p)$ to write
\begin{equation}\label{Fpsw:p>2}
\mathcal{F}_p(s, w)
= \prod_{0\le j\le 2} \big(1-p^{-(s+2jw-2j)}\big)^{-1} \mathcal{G}_p(s, w),
\end{equation}
where
\begin{equation}\label{def:Gpsw}
\begin{aligned}
& \mathcal{G}_p(s, w)
\\
& := \bigg(1
+\frac{p^2+p+1}{p^{s+2w}}
+\frac{p^3+p^2+p}{p^{s+4w}}
+\frac{p^3}{p^{2s+6w}}\bigg)
\bigg(1-\frac{p^2}{p^{s+2w}}\bigg)
\bigg(1-\frac{1}{p^{s+4w}}\bigg)^{-1}.
\end{aligned}
\end{equation}
While for $p=2$, the formula \eqref{Formal_Calcul_2} 
with $(x, y, z, a) = (2^{-s}, 2^{-w}, 2, 3)$ gives 
\begin{equation}\label{Fpsw:p=2}
\mathcal{F}_2(s, w)
= \prod_{0\le j\le 2} \big(1-2^{-(s+2jw-2j)}\big)^{-1} \mathcal{G}_2(s, w),
\end{equation}
where
\begin{equation}\label{def:G2sw}
\mathcal{G}_2(s, w)
:= \frac{1+3\cdot 2^{-s-2w}+2^{-s-4w+1}}{1-2^{-s-4w}} \prod_{1\le j\le 2} (1-2^{-(s+2jw-2j)}).
\end{equation}
Combining \eqref{Fpsw:p>2}--\eqref{def:G2sw}, we get \eqref{Expression:Fsw} with
\begin{equation}\label{def:Gsw}
\mathcal{G}(s, w) := \prod_p \mathcal{G}_p(s, w)
\qquad
(\re s>5, \; \re w>0).
\end{equation}
It is easy to verify that
for $\min_{0\le j\le 2} (\sigma+2ju-2j)\ge \tfrac{1}{2}+\varepsilon$, we have
$|\mathcal{G}_p(s, w)|= 1 + O(p^{-1-\varepsilon})$.
This shows that under the same condition, the Euler product $\mathcal{G}(s, w)$ converges absolutely
and \eqref{UB:Gsw} holds.
By analytic continuation, \eqref{Expression:Fsw} is also true in the same domain.
This completes the proof.
\end{proof}

\section{Proof of Theorem \ref{thm2}}\label{PfThm2}

In the sequel, we suppose 
\begin{equation}\label{Condition:XYTUHJ}
10\le X\le Y\le X^3,
\quad
(XY)^3\le 4T\le U\le X^{12},
\quad
H\le X,
\quad
J\le Y,
\end{equation}
and for brevity we fix the following notation:
\begin{equation}\label{def:kappa_Lambda_L}
s := \sigma+\mathrm{i}\tau,
\quad
w := u+\mathrm{i}v,
\quad
\mathcal{L} := \log X,
\quad
\kappa := 1+\mathcal{L}^{-1},
\quad
\lambda := 1+4\mathcal{L}^{-1}.
\end{equation}
The following proposition is an immediate consequence of 
Lemmas \ref{Perron_Formula:MXY}-\ref{Lem:Evaluate_I2} below. 

\begin{proposition}\label{Pro:M1XY}
Under the previous notation, we have
$$
M(X, Y)
= X^2 Y^2 P (\log X-\tfrac{1}{4}\log Y) + R_0(X, Y) + R_1(X, Y) + R_2(X, Y) + O(1)
$$
uniformly for $(X, Y, T, U, H, J)$ satisfying 
\eqref{Condition:XYTUHJ},
where $R_0, R_1, R_2$ and $P(t)$ are defined as in 
\eqref{def:R0XY}, 
\eqref{def:R1XY}, 
\eqref{def:R2XY}
and \eqref{def:Pt} below,
respectively.
\end{proposition}

The proof is divided into several subsections. 

\subsection{Application of Perron's formula} 
The first step is to apply Perron's formula twice to transform 
$M(X, Y)$ into a form that is ready for future treatment.  

\begin{lemma}\label{Perron_Formula:MXY}
Under the previous notation, we have
\begin{equation}\label{MXY=MXYTU}
M(X, Y)
= M(X, Y; T, U) + O(1)
\end{equation}
uniformly for $(X, Y, T, U)$ satisfying 
\eqref{Condition:XYTUHJ}, where the implied constant is absolute and
\begin{equation}\label{def:M1XYTU}
M(X, Y; T, U)
:= \frac{1}{(2\pi {\rm i})^2} \int_{\kappa-{\rm i}T}^{\kappa+{\rm i}T} 
\bigg(\int_{\lambda-{\rm i}U}^{\lambda+{\rm i}U} \frac{\mathcal{F}(s, w) Y^{w+1}}{w(w+1)} \d w\bigg) 
\frac{X^{s+1}}{s(s+1)} \d s.
\end{equation}
\end{lemma}

The proof is the same as that of \cite[Lemma 6.2]{LWZ2017}.

\subsection{Application of Cauchy's theorem} 
In this subsection, we shall apply Cauchy's theorem to 
evaluate the integral over $w$ in $M(X, Y; T, U)$.
We write
\begin{equation}\label{def:wjs} 
w_j = w_j(s) := (2j+1-s)/(2j)
\quad
(1\le j\le 2) 
\end{equation} 
and
\begin{equation}\label{def:Fk*} 
\mathcal{F}_1^*(s)
:= \zeta(s) \zeta(2-s) \mathcal{G}(s, w_1(s)),
\qquad
\mathcal{F}_2^*(s)
:= \zeta(s) \zeta(\tfrac{s+1}{2}) \mathcal{G}(s, w_2(s)).
\end{equation}

\begin{lemma}\label{Lem:MXYTU}
Under the previous notation, for any $\varepsilon>0$ we have
\begin{equation}\label{Evaluate:M1XYTU}
M(X, Y; T, U)
= I_1 + I_2 + R_0(X, Y) + O_{\varepsilon}(1)
\end{equation}
uniformly for $(X, Y, T, U)$ satisfying \eqref{Condition:XYTUHJ}, 
where 
\begin{align*}
I_1
& :=  \frac{4}{2\pi {\rm i}} \int_{\kappa-{\rm i}T}^{\kappa+{\rm i}T} 
\frac{\mathcal{F}_1^*(s) X^{s+1} Y^{(5-s)/2}}{(3-s)(5-s)s(s+1)} \d s, 
\\\noalign{\vskip 1mm}
I_2
& :=  \frac{16}{2\pi {\rm i}} 
\int_{\kappa-{\rm i}T}^{\kappa+{\rm i}T} \frac{\mathcal{F}_2^*(s) X^{s+1}Y^{(9-s)/4}}{(5-s)(9-s)s(s+1)} \d s, 
\end{align*}
and
\begin{equation}\label{def:R0XY}
R_{ 0}(X, Y)
:= \frac{1}{(2\pi\mathrm{i})^2} \int_{\kappa-{\rm i}T}^{\kappa+{\rm i}T} 
\bigg(\int_{\frac{11}{12}+\varepsilon-{\rm i}U}^{\frac{11}{12}+\varepsilon+{\rm i}U} \frac{\mathcal{F}(s, w) Y^{w+1}}{w(w+1)} \d w\bigg)
\frac{X^{s+1}}{s(s+1)} \d s.
\end{equation}
Furthermore we have
\begin{equation}\label{UB_TR0}
\left.
\begin{array}{rl}
(\mathscr{D}R_{ 0})(X, X+H; Y, Y+J)\!
\\\noalign{\vskip 1mm}
(\mathscr{D}R_{ 0})(X-H, X; Y-J, Y)\!
\end{array}
\right\}
\ll_{\varepsilon} X^{\frac{7}{6}+\varepsilon} Y^{\frac{11}{12}+\varepsilon} H^{\frac{5}{6}} J 
+ X^{1+\varepsilon} Y^{\frac{13}{12}+\varepsilon} H J^{\frac{5}{6}}
\end{equation}
uniformly for $(X, Y, T, U, H, J)$ satisfying 
\eqref{Condition:XYTUHJ}.
\end{lemma}

\begin{proof}
We want to calculate the integral
$$
\frac{1}{2\pi {\rm i}} \int_{\lambda-{\rm i}U}^{\lambda+{\rm i}U}
\frac{\mathcal{F}(s, w) Y^{w+1}}{w(w+1)} \d w
$$
for any individual $s=\sigma+\mathrm{i}\tau$ with $\sigma=\kappa$ and $|\tau|\le T$. 
We move the line of integration $\re w = \lambda$ to $\re w=\tfrac{3}{4}+\varepsilon$.
By Lemma \ref{Lem:Fsw},  for $\sigma=\kappa$ and $|\tau|\le T$,
the points $w_j(s) \; (j=1, 2)$, given by \eqref{def:wjs},
are the simple poles of the integrand in the rectangle 
$\tfrac{3}{4}+\varepsilon\le u\le \lambda$ and $|v|\le U$.
The residues of $\frac{\mathcal{F}(s, w)}{w(w+1)} Y^{w+1}$ at the poles $w_j(s)$ are
\begin{equation}\label{def:residue}
\frac{4\mathcal{F}_1^*(s)Y^{(5-s)/2}}{(3-s)(5-s)},
\qquad
\frac{16\mathcal{F}_2^*(s)Y^{(9-s)/4}}{(5-s)(9-s)}, 
\end{equation}
respectively, where $\mathcal{F}_j^*(s) (j=1, 2)$ are defined as in \eqref{def:Fk*}.

It is well-known that 
(cf. e.g. \cite[page 146, Theorem II.3.7]{Tenenbaum1995}) 
\begin{equation}\label{UB:zeta}
\zeta(s)\ll |\tau|^{\max\{(1-\sigma)/3, 0\}} \log |\tau|
\qquad
(\sigma\ge \tfrac{1}{2}, \; |\tau|\ge 2) 
\end{equation}
where $c>0$ is a constant. 
When $\sigma=\kappa$ and $\tfrac{11}{12}+\varepsilon\le u\le \lambda$, 
it is easily checked that 
$$
\min_{0\le j\le 2} (\sigma+2ju-2j)
\ge 1+4(\tfrac{11}{12}+\varepsilon-1)
=\tfrac{3}{4}+3\varepsilon
>\tfrac{1}{2}+\varepsilon.
$$
It follows from \eqref{UB:zeta} and \eqref{UB:Gsw} that  
$\mathcal{F}(s, w)\ll_{\varepsilon} U^{2(1-u)}\mathcal{L}^4$
for $\sigma=\kappa, |\tau|\le T, 
\tfrac{11}{12}+\varepsilon\le u\le \lambda$ and $v=\pm U$.
This implies that
$$
\int_{\frac{11}{12}+\varepsilon\pm{\rm i}U}^{\lambda\pm{\rm i}U} \frac{\mathcal{F}(s, w) Y^{w+1}}{w(w+1)} \d w 
\ll_{\varepsilon} Y\mathcal{L}^4 \int_{\frac{11}{12}}^{\lambda}\bigg(\frac{Y}{U^2}\bigg)^u \d u
\ll_{\varepsilon} \frac{Y^{\frac{23}{12}}\mathcal{L}^4}{U^{\frac{11}{6}}}
\ll_{\varepsilon} 1.
$$
Cauchy's theorem then gives 
\begin{align*}
\frac{1}{2\pi {\rm i}} 
\int_{\lambda-{\rm i}U}^{\lambda+{\rm i}U} \frac{\mathcal{F}(s, w) Y^{w+1}}{w(w+1)} \d w
& = \frac{4\mathcal{F}_1^*(s)Y^{(5-s)/2}}{(3-s)(5-s)}
+ \frac{16\mathcal{F}_2^*(s)Y^{(9-s)/4}}{(5-s)(9-s)} 
\\
& \quad
+ \frac{1}{2\pi {\rm i}} \int_{\frac{11}{12}+\varepsilon-{\rm i}U}^{\frac{11}{12}+\varepsilon+{\rm i}U} 
\frac{\mathcal{F}(s, w) Y^{w+1}}{w(w+1)} \d w
+ O_{\varepsilon}(1).
\end{align*}
Inserting the last formula 
into \eqref{def:M1XYTU}, we obtain \eqref{Evaluate:M1XYTU}.

Finally we prove \eqref{UB_TR0}.
For $\sigma=\kappa$, $|\tau|\le T$, $u=\tfrac{11}{12}+\varepsilon$ and $|v|\le U$, 
we apply \eqref{UB:zeta} and \eqref{UB:Gsw} as before, to get 
$$
\mathcal{F}(s, w)
\ll (|\tau|+|v|+1)^{\frac{1}{6}}\mathcal{L}^4
\ll \big\{(|\tau|+1)^{\frac{1}{6}} + (|v|+1)^{\frac{1}{6}}\big\}\mathcal{L}^4. 
$$
Also, for $\sigma, \tau, u, v$ as above, we have 
\begin{align*}
r_{s, w}(X, H; Y, J)
& := \big((X+H)^{s+1}-X^{s+1}\big) \big((Y+J)^{w+1}-Y^{w+1}\big)
\\
& \ll X^2((|\tau|+1)H/X))^{\frac{5}{6}-\varepsilon} Y^{\frac{23}{12}+\varepsilon} ((|v|+1)J/Y)^{1-\varepsilon}
\\
& \ll X^{\frac{7}{6}+\varepsilon} Y^{\frac{11}{12}+\varepsilon} H^{\frac{5}{6}} J (|\tau|+1)^{\frac{5}{6}-\varepsilon} (|v|+1)^{1-\varepsilon}
\end{align*}
by \eqref{Lem3.2_Eq_A} of Lemma \ref{Lem3.2} with 
$\beta=\tfrac{5}{6}-\varepsilon$ and with $\beta=1-\varepsilon$.
Similarly, 
\begin{align*}
r_{s, w}(X, H; Y, J)
& = \big((X+H)^{s+1}-X^{s+1}\big) \big((Y+J)^{w+1}-Y^{w+1}\big)
\\
& \ll X^2((|\tau|+1)H/X))^{1-\varepsilon} Y^{\frac{23}{12}+\varepsilon} ((|v|+1)J/Y)^{\frac{5}{6}-\varepsilon}
\\
& \ll X^{1+\varepsilon} Y^{\frac{13}{12}+\varepsilon} H J^{\frac{5}{6}} 
(|\tau|+1)^{1-\varepsilon} (|v|+1)^{\frac{5}{6}-\varepsilon}
\end{align*}
by Lemma \ref{Lem3.2} with $\beta=1-\varepsilon$ 
and with  $\beta=\tfrac{5}{6}-\varepsilon$. 
These and Lemma \ref{Operator/2}(i) imply 
\begin{align*}
(\mathscr{D}R_0)(X, X+H; Y, Y+J)
& = \int_{\kappa-{\rm i}T}^{\kappa+{\rm i}T} \int_{\frac{11}{12}+\varepsilon-{\rm i}U}^{\frac{11}{12}+\varepsilon+{\rm i}U} 
\frac{\mathcal{F}(s, w) }{(2\pi\mathrm{i})^2}
\frac{r_{s, w}(X, H; Y, J)}{s(s+1)w(w+1)} \d w \d s
\\\noalign{\vskip 1mm}
& \ll_{\varepsilon} X^{\frac{7}{6}+\varepsilon} Y^{\frac{11}{12}+\varepsilon} H^{\frac{5}{6}} J 
+ X^{1+\varepsilon} Y^{\frac{13}{12}+\varepsilon} H J^{\frac{5}{6}}.
\end{align*}
This completes the proof.
\end{proof}

\subsection{Evaluation of $I_1$}\

\vskip 1mm

\begin{lemma}\label{Lem:Evaluate_I1}
Under the previous notation, we have
\begin{equation}\label{Evaluate:I1}
I_1 = R_1(X, Y) + O(1)
\end{equation}
uniformly for $(X, Y, T)$ satisfying \eqref{Condition:XYTUHJ},
where 
\begin{equation}\label{def:R1XY}
R_1(X, Y)
:= \frac{4}{2\pi {\rm i}} \int_{\frac{5}{4}-{\rm i}T}^{\frac{5}{4}+{\rm i}T} 
\frac{\mathcal{F}_1^*(s)X^{s+1}Y^{(5-s)/2}}{(3-s)(5-s)s(s+1)} \d s.
\end{equation}
Further we have
\begin{equation}\label{UB_TR1}
\left.
\begin{array}{rl}
(\mathscr{D}R_1)(X, X+H; Y, Y+J)\!
\\\noalign{\vskip 1mm}
(\mathscr{D}R_1)(X-H, X; Y-J, Y)\!
\end{array}\right\}
\ll X^{\frac{5}{4}} Y^{\frac{7}{8}} H J
\end{equation}
uniformly for $(X, Y, T, H, J)$ satisfying \eqref{Condition:XYTUHJ}.
\end{lemma}

\begin{proof}
We shall prove \eqref{Evaluate:I1} by moving the contour $\re s = \kappa$ to $\re s=\tfrac{5}{4}$. 
When $\kappa\le \sigma\le \tfrac{5}{4}$, it is easy to check that
$$
\min_{0\le j\le 2} (\sigma+2jw_1(\sigma)-2j)
= \min_{0\le j\le 2} (j+(1-j)\sigma)
\ge \tfrac{3}{4}\cdot
$$
By Lemma \ref{Lem:Fsw} the integrand is holomorphic in the rectangle 
$\kappa\le \sigma\le \tfrac{5}{4}$ and $|\tau|\le T$; 
and we can apply \eqref{UB:zeta} and \eqref{UB:Gsw} to get $\mathcal{F}_1^*(s)\ll T^{(\sigma-1)/3} \mathcal{L}^2$
in this rectangle, 
which implies that
\begin{align*}
\int_{\kappa\pm{\rm i}T}^{\frac{5}{4}\pm{\rm i}T}
\frac{\mathcal{F}_1^*(s)X^{s+1}Y^{(5-s)/2}}{(3-s)(5-s)s(s+1)} \d s
& \ll \frac{X^2Y^2\mathcal{L}^2}{T^4} 
\int_{\kappa}^{\frac{5}{4}} \bigg(\frac{XT^{1/3}}{Y^{1/2}}\bigg)^{\sigma-1} \d \sigma
\\
& \ll \frac{X^2Y^2\mathcal{L}^2}{T^4}+\frac{X^{\frac{9}{4}}Y^{\frac{15}{8}}\mathcal{L}^2}{T^{\frac{47}{12}}}
\ll 1.
\end{align*}
This proves \eqref{Evaluate:I1}. 

To establish \eqref{UB_TR1}, we note that $\mathcal{F}_1^*(s)\ll (|\tau|+1)^{\frac{1}{4}}$
for $\sigma=\frac{5}{4}$ and $|\tau|\le T$.
By \eqref{Lem3.2_Eq_A} of Lemma \ref{Lem3.2} with $\beta=1$, 
\begin{align*}
r_{s, w_1(s)}(X, H; Y, J)
& := \big((X+H)^{s+1}-X^{s+1}\big) \big((Y+J)^{(5-s)/2}-Y^{(5-s)/2}\big)
\\
& \ll X^{\frac{5}{4}} Y^{\frac{7}{8}} H J (|\tau|+1)^2. 
\end{align*}
Combining these with Lemma \ref{Operator/2}(ii), we deduce that 
\begin{align*}
(\mathscr{D}R_1)(X, X+H; Y, Y+J)
& = \frac{4}{2\pi\text{i}} \int_{\frac{5}{4}-{\rm i}T}^{\frac{5}{4}+{\rm i}T} 
\frac{\mathcal{F}_1^*(s)r_{s, w_1(s)}(X, H; Y, J)}{(3-s)(5-s)s(s+1)} \d s
\\
& \ll X^{\frac{5}{4}} Y^{\frac{7}{8}} H J, 
\end{align*}
from which the desired result follows. 
\end{proof}

\subsection{Evaluation of $I_2$}\

\vskip 1mm

\begin{lemma}\label{Lem:Evaluate_I2}
Under the previous notation, for any $\varepsilon>0$ we have
\begin{equation}\label{Evaluate:I2}
I_2 = X^2Y^2P(\log X-\tfrac{1}{4}\log Y) + R_2(X, Y) + O_{\varepsilon}(1)
\end{equation}
uniformly for $(X, Y, T)$ satisfying \eqref{Condition:XYTUHJ},
where $P(t)$ is defined as in \eqref{def:Pt} below and
\begin{equation}\label{def:R2XY}
R_2(X, Y)
:= \frac{16}{2\pi {\rm i}} \int_{\frac{1}{2}+\varepsilon-{\rm i}T}^{\frac{1}{2}+\varepsilon+{\rm i}T} 
\frac{\mathcal{F}_2^*(s)X^{s+1}Y^{(9-s)/4}}{(5-s)(9-s)s(s+1)} \d s.
\end{equation}
Further we have
\begin{equation}\label{UB_TR2}
\left.
\begin{array}{rl}
(\mathscr{D}R_2)(X, X+H; Y, Y+J)\!
\\\noalign{\vskip 1mm}
(\mathscr{D}R_2)(X-H, X; Y-J, Y)\!
\end{array}\right\}
\ll_{\varepsilon} X^{\frac{1}{2}+\varepsilon} Y^{\frac{9}{8}} HJ
\end{equation}
uniformly for $(X, Y, T, H, J)$ satisfying \eqref{Condition:XYTUHJ}.
\end{lemma}

\begin{proof}
We move the line of integration $\re s = \kappa$ to $\re s=\tfrac{1}{2}+\varepsilon$.
Obviously $s=1$  
is the unique pole of order 2 of the integrand in the rectangle 
$\tfrac{1}{2}+\varepsilon\le \sigma\le \kappa$ and $|\tau|\le T$, and 
the residue is $X^2Y^2P(\log X-\tfrac{1}{4}\log Y)$ with 
\begin{equation}\label{def:Pt}
P(t)
:= \bigg(\frac{16(s-1)^2\mathcal{F}_2^*(s) \mathrm{e}^{t(s-1)}}{(5-s)(9-s)s(s+1)}\bigg)'\bigg|_{s=1}.
\end{equation}
Here $P(t)$ is a linear polynomial with the leading coefficient $\mathscr{C}_4$ given by \eqref{def:LeadingCoefficient}
above.

When $\tfrac{1}{2}+\varepsilon\le \sigma\le \kappa$, 
we check that 
$$
\min_{0\le j\le 2} (\sigma+2jw_2(\sigma)-2j)
= \tfrac{1}{2} \min_{0\le j\le 2} (j+(2-j)\sigma)
\ge \tfrac{1}{2}+\varepsilon.
$$
Hence when $\tfrac{1}{2}+\varepsilon\le \sigma\le \kappa$ and $|\tau|\le T$,
\eqref{UB:zeta} and \eqref{UB:Gsw} yields $\mathcal{F}_2^*(s)\ll T^{(1-\sigma)/2} \mathcal{L}^3$.
It follows that 
\begin{align*}
\int_{\frac{1}{2}+\varepsilon\pm{\rm i}T}^{\kappa\pm{\rm i}T} 
\frac{\mathcal{F}_2^*(s)X^{s+1}Y^{(9-s)/4}}{(5-s)(9-s)s(s+1)} \d s
& \ll \frac{X^2Y^2\mathcal{L}^3}{T^4}
\int_{\frac{1}{2}}^{\kappa} 
\bigg(\frac{YT^2}{X^4}\bigg)^{(1-\sigma)/4} \d s
\\
& \ll \frac{X^2Y^2\mathcal{L}^3}{T^4}+\frac{X^{\frac{3}{2}}Y^{\frac{9}{4}}\mathcal{L}^3}{T^{\frac{7}{2}} }
\ll 1.
\end{align*}
These establish \eqref{Evaluate:I2}. To prove \eqref{UB_TR2}, 
we note that for $\sigma=\tfrac{1}{2}+\varepsilon$ and $|\tau|\le T$, we have 
$
\mathcal{F}_2^*(s)
\ll_{\varepsilon} (|\tau|+1)^{1/3}
$
thanks to \eqref{UB:zeta} and \eqref{UB:Gsw}, and
\begin{align*}
r_{s, w_2(s)}(X, H; Y, J)
& := \big((X+H)^{s+1}-X^{s+1}\big)\big((Y+J)^{(9-s)/4}-Y^{(9-s)/4}\big)
\\
& \ll_{\varepsilon} X^{\frac{1}{2}+\varepsilon} Y^{\frac{9}{8}} HJ (|\tau|+1)^2
\end{align*}
by Lemma \ref{Lem3.2} with $\beta=1$.
Combining these with Lemma \ref{Operator/2}(i), we deduce that
\begin{align*}
(\mathscr{D}R_2)(X, X+H; Y, Y+J)
& = \frac{16}{2\pi\text{i}} \int_{\frac{1}{2}+\varepsilon-{\rm i}T}^{\frac{1}{2}+\varepsilon+{\rm i}T}  
\frac{\mathcal{F}_2^*(s) r_{s, w_2(s)}(X, H; Y, J)}{(5-s)(9-s)s(s+1)} \d s
\\
& \ll_{\varepsilon} X^{\frac{1}{2}+\varepsilon} Y^{\frac{9}{8}} HJ.
\end{align*}
This proves the lemma. 
\end{proof} 

\subsection{Completion of proof of Theorem \ref{thm2}} 
Denote by $\mathcal{M}(X, Y)$ the main term in the asymptotic formula 
of $M(x, y)$ in Proposition~\ref{Pro:M1XY}, that is 
$\mathcal{M}(X, Y) := X^2 Y^2 P(\psi)$ and $\psi := \log(X/Y^{1/4})$. 
Then Lemma \ref{Operator/2}(i) gives 
\begin{align*}
(\mathscr{D}\mathcal{M})(X, X+H; Y, Y+J)
& = \big\{XY \big(4P(\psi) + \tfrac{3}{2}P'(\psi)\big) + O(XJ\mathcal{L}^2+YH\mathcal{L}^2)\big\}HJ.
\end{align*}
Since $\mathscr{D}$ is a linear operator, 
this together with Proposition~\ref{Pro:M1XY}
implies that
\begin{align*}
(\mathscr{D}M)(X, X+H; Y, Y+J)
= \big\{XY \big(4P(\psi) + \tfrac{3}{2}P'(\psi)\big) + O_{\varepsilon}(\mathcal{R})\big\}HJ
\end{align*}
with
\begin{align*}
\mathcal{R}
:= X^{\frac{7}{6}+\varepsilon} Y^{\frac{11}{12}} H^{-\frac{1}{6}} 
+ X^{1+\varepsilon} Y^{\frac{13}{12}} J^{-\frac{1}{6}}
+ X^{\frac{5}{4}} Y^{\frac{7}{8}}
+ X^{\frac{1}{2}+\varepsilon} Y^{\frac{9}{8}}
+ XJ\mathcal{L}^2
+ YH\mathcal{L}^2.
\end{align*}
The same formula also holds for $(\mathscr{D}M)(X-H, X; Y-J, Y)$.
Now Lemma~ \ref{Operator/2}(ii) with 
$H=XY^{-\frac{1}{14}}$ and $J=Y^{\frac{13}{14}}$ allows us to deduce
$$
S(X, Y) 
= XY \big(4P(\psi) + \tfrac{3}{2}P'(\psi)\big) 
+ O_{\varepsilon}\big(X^{\frac{5}{4}} Y^{\frac{7}{8}} 
+ X^{\frac{1}{2}+\varepsilon} Y^{\frac{9}{8}}\big),
$$
where we have used the following facts  
\begin{align*}
(X^{\frac{5}{4}} Y^{\frac{7}{8}})^{\frac{5-24\varepsilon}{6}} 
(X^{\frac{1}{2}+\varepsilon} Y^{\frac{9}{8}})^{\frac{1+24\varepsilon}{6}}
& = X^{\frac{27}{24}-\frac{4(17-24\varepsilon)\varepsilon}{24}} Y^{\frac{11}{12}+\varepsilon}
\ge X^{1+\varepsilon} Y^{\frac{11}{12}+\varepsilon},
\\
(X^{\frac{5}{4}} Y^{\frac{7}{8}})^{\frac{11}{14}} 
(X^{\frac{1}{2}+\varepsilon} Y^{\frac{9}{8}})^{\frac{3}{14}}
& = X^{\frac{61+12\varepsilon}{56}} Y^{\frac{13}{14}}
\ge X^{1+\varepsilon} Y^{\frac{13}{14}}.
\end{align*}
This finally completes the proof of Theorem~\ref{thm2}.

\section{Proof of Theorems~\ref{thm3} and \ref{thm1}}\label{PfThm3}

\begin{proof}[Proof of Theorems~\ref{thm3}]  
The idea is to apply Theorems~\ref{thm2} in a delicate way. 
Trivially we have $r_4^*(d)\le d\tau(d)$ (here $\tau(n)$ is the divisor function), and therefore 
\begin{equation}\label{UB:Sxy}
S(x, y)
\le y\sum_{n\le x} \sum_{d\mid n^4} \tau(d)
\ll xy(\log x)^{14}
\end{equation}
for all $x\ge 2$ and $y\ge 2$, where the implied constant is absolute.

Let $\delta := 1-(\log B)^{-1}$ and let $k_0$ be a positive integer such that
$\delta^{k_0}<(\log B)^{-3}\le \delta^{k_0-1}$.
Note that $k_0\asymp (\log B)\log\log B$. 
In view of \eqref{UB:Sxy}, we can write
\begin{equation}\label{UB:TB}
\begin{aligned}
T(B)
& = \sum_{1\le k\le k_0} \sum_{\delta^kB<n\le \delta^{k-1}B} 
\sum_{\substack{d\mid n^4\\ d<n^4/B^2}} 
r_4^*(d) \mathbb{1}_{\square}\bigg(\frac{n^4}{d}\bigg)
+ O(B^3)
\\
& \le \sum_{1\le k\le k_0} \big(S(\delta^{k-1}B, \delta^{4(k-1)}B^2) - S(\delta^{k}B, \delta^{4(k-1)}B^2)\big)
+ O(B^3).
\end{aligned}
\end{equation}
Similarly (even easily), 
\begin{equation}\label{LB:TB}
\begin{aligned}
T(B)
& \ge \sum_{1\le k\le k_0} \big(S(\delta^{k-1}B, \delta^{4k}B^2) - S(\delta^{k}B, \delta^{4k}B^2)\big).
\end{aligned}
\end{equation}
By \eqref{Cor:Sxy} of Theorem \ref{thm2} with $\eta=\tfrac{1}{4}$, a simple computation shows that
\begin{align}
T(B)
& \le 2(1-\delta)\frac{1-\delta^{5k_0}}{1-\delta^5} \cdot \mathscr{C}_4 B^3\log B\,
\Big\{1+ O\Big(\frac{1}{\sqrt[4]{\log B}}\Big)\Big\} + O(B^3),
\label{UB-Final:TB}
\\
T(B)
& \ge 2(\delta^{-1}-1)\frac{\delta^{5}-\delta^{5(k_0+1)}}{1-\delta^5} \cdot
\mathscr{C}_4 B^3\log B\,
\bigg\{1+ O\bigg(\frac{1}{\sqrt[4]{\log B}}\bigg)\bigg\}.
\label{LB-Final:TB}
\end{align}
By noticing that
\begin{align*}
(1-\delta)\frac{1-\delta^{5k_0}}{1-\delta^5}
& = \frac{1-\delta^{5k_0}}{1+\delta+\delta^2+\delta^3}
= \frac{1}{5} + O\bigg(\frac{1}{\sqrt[4]{\log B}}\bigg),
\\
(\delta^{-1}-1)\frac{\delta^{5}-\delta^{5(k_0+1)}}{1-\delta^5}
& = \frac{\delta^{4}-\delta^{5k_0+4}}{1+\delta+\delta^2+\delta^3+\delta^4}
= \frac{1}{5} + O\bigg(\frac{1}{\sqrt[4]{\log B}}\bigg).
\end{align*}
The desired asymptotic formula \eqref{Evaluation:TB} follows from \eqref{UB-Final:TB} and \eqref{LB-Final:TB}. 
\end{proof} 

\begin{proof}[Proof of Theorem~\ref{thm1}] 
Applying \eqref{Evaluation:Sxy} of Theorem \ref{thm2} with $(x, y)=(B, B^2)$, we have 
\begin{equation}\label{Proof:thm1_3}
\sum_{n\le B} \sum_{\substack{d\mid n^4\\ d\le B^2}} r_4^*(d) 
\mathbb{1}_{\square}\bigg(\frac{n^4}{d}\bigg)
= 2 \mathscr{C}_4 B^3\log B \, \bigg\{1+O\bigg(\frac{1}{\sqrt[4]{\log B}}\bigg)\bigg\}.
\end{equation}
Inserting this and \eqref{Evaluation:TB} into \eqref{decomposition:N4B},
we obtain \eqref{eq:N4*B} with $\mathcal{C}_4^*= \tfrac{192}{5}\mathscr{C}_4$.

Finally \eqref{eq:N4B} follows from \eqref{eq:N4*B} via the inversion formula of M\"obius. 
\end{proof} 

\noindent 
{\bf Acknowledgements.} 
This work was supported by National Natural Science Foundation of China (Grant Nos.
11531008 and 11771121) and the Program PRC 1457-AuForDiP (CNRS-NSFC).

\end{document}